\documentclass[oneside]{amsart}

\usepackage{amssymb}
\usepackage{amsthm}
\usepackage{amscd}
\usepackage{enumitem}

\numberwithin{equation}{section}
\theoremstyle{plain}

\newtheorem{thm}{Theorem}[section]
 
 \newtheorem{lemma}[thm]{Lemma}
\newtheorem{prop}[thm]{Proposition}

\theoremstyle{definition}

\newcommand{\dlabel}[1]{\ifmmode \text{\ttfamily \upshape [#1] } \else
{\ttfamily \upshape [#1] }\fi \label{#1}}

\newcommand{\Hom}{\operatorname{Hom} }

\begin{document}

\setlength{\baselineskip}{15pt}

\title{Classification of $p$-groups by their Schur Multiplier}

\author{Sumana Hatui}
\address{School of Mathematics, Harish-Chandra Research Institute, Chhatnag Road, Jhunsi, Allahabad 211019, INDIA}
\email{sumanahatui@hri.res.in, sumana.iitg@gmail.com}


\begin{abstract}
Let $G$ be a non-abelian $p$-group of order $p^n$ and $M(G)$ be its Schur multiplier. It is well known result by Green that $|M(G)| \leq p^{\frac{1}{2}n(n-1)}$. So $|M(G)|= p^{\frac{1}{2}n(n-1)-t(G)}$ for some $t(G) \geq 0$. The groups has already been classified for $t(G) \leq 5$ by several authors. For $t(G)=6$ the classification has been done in \cite{SHJ}. In this paper we classify $p$-groups $G$ for $t(G) = 6$ in different method.
\end{abstract}

\maketitle

\section{Introduction}
The Schur multiplier $M(G)$ of a group $G$ was introduced by Schur \cite{IS1} in 1904 on the study of projective representation of groups. 

For $p$-groups $G$ of order $p^n$, Green \cite{JG} gave an upper bound $p^{\frac{1}{2}n(n-1)}$ for order of the Schur Multiplier $M(G)$. So we have $|M(G)|=p^{\frac{1}{2}n(n-1)-t(G)}$, for some $t(G) \geq 0$. Now the question comes in our mind that whether it is possible to classify the structure of all $p$-groups $G$ by the order of the Schur multiplier $M(G)$, i.e., when $t(G)$ is known. Several authors have already answered this question. They classified the groups of order $p^n$ for $t(G) \leq 5$ in \cite{BY,ZH,EG,PN3,PN1}. The structure of $p$-groups with $t(G) = 6$ has been determined in \cite{SHJ}. In the present paper,  we classify
the structure of all non-abelian finite $p$-groups when $t(G) = 6$, i.e. $|M(G)|=p^{\frac{1}{2}n(n-1)-6}$. Our method is quite different to that of \cite{SHJ}. We have stated some structural results of group $G$ with the assumtion $t(G)=6$.

By $ES_p(p^3)$ and $ES_{p^2}(p^3)$ we denote extra-special $p$-groups of order $p^3$ having exponent $p$ and $p^2$ respectively. By $ES_p(p^5)$ and $ES_{p^2}(p^5)$ we denote extra-special $p$-groups of order $p^5$ having exponent $p$ and $p^2$ respectively. By $\mathbb{Z}_p^{(k)}$ we denote $\mathbb{Z}_p \times \mathbb{Z}_p \cdots \times \mathbb{Z}_p$($k$ times).

James \cite{RJ} classified all $p$-groups of order $p^n$ for $n \leq 6$ upto isoclinism which are denoted by $\Phi_k$. We use his notation throughout this paper.

In this paper we prove the following result. 
\begin{thm} (Main Theorem)
Let $G$ be a non-abelian $p$-group of order $p^n$ with $|M(G)|=p^{\frac{1}{2}n(n-1)-6}$.\\ 
If $p$ is odd, then $G$ is isomorphic to\\
(i) $ES_p(p^3) \times \mathbb{Z}_p^{(5)}$.\\
(ii) $\Phi_2(21^4)a = ES_{p^2}(p^3) \times \mathbb{Z}_p^{(3)}$.\\
(iii) $\Phi_2(21^4)b= \Phi_2(211)b \times \mathbb{Z}_p^{(2)}$, \\
where $\Phi_2(211)b = \langle{\alpha,\alpha_1,\alpha_2,\gamma \mid [\alpha_1,\alpha]=\gamma^p=\alpha_2, \alpha^p=\alpha_1^p=\alpha_2^p=1\rangle}$.\\
(iv) $\Phi_5(21^4)a = ES_{p^2}(p^5) \times \mathbb{Z}_p$.\\
(v) $\Phi_5(1^6)=ES_p(p^5) \times \mathbb{Z}_p$.\\
(vi) $\Phi_5(21^4)b=\langle{\alpha_1,\alpha_2,\alpha_3,\alpha_4,\beta,\gamma \mid [\alpha_1,\alpha_2]=[\alpha_3,\alpha_4]=\gamma^p=\beta,\alpha_i^p=\beta^p=1(i=1,2,3,4)\rangle}$.\\
(vii) $\Phi_4(1^6)=\Phi_4(1^5) \times \mathbb{Z}_p$, \\
where $\Phi_4\left(1^5\right)= \langle{\alpha,\alpha_1,\alpha_2,\beta_1,\beta_2 \mid [\alpha_i,\alpha]=\beta_i, \alpha^p=\alpha_i^p=\beta_i^p=1 (i=1,2)\rangle}$.\\
(viii) $\Phi_2(2111)c=\Phi_2(211)c \times \mathbb{Z}_p$,\\
where $\Phi_2(211)c = \langle{\alpha,\alpha_1,\alpha_2 \mid [\alpha_1,\alpha]=\alpha_2, \alpha^{p^2}=\alpha_1^p=\alpha_2^p=1\rangle}$.\\
(ix) $\Phi_2(2111)d=ES_p(p^3) \times \mathbb{Z}_{p^2}$.\\
(x) $\Phi_3(1^5)=\Phi_3(1^4) \times \mathbb{Z}_p$, \\
where $\Phi_3(1^4) = \langle{\alpha,\alpha_1,\alpha_2,\alpha_3 \mid [\alpha_i,\alpha]=\alpha_{i+1},\alpha^p=\alpha_i^{(p)}=\alpha_3^p=1(i=1,2)\rangle}$.\\
(xi) $\Phi_7(1^5)=\langle{\alpha,\alpha_1,\alpha_2,\alpha_3,\beta \mid [\alpha_i,\alpha]=\alpha_{i+1},[\alpha_1,\beta]=\alpha_3, \alpha^p=\alpha_1^{(p)}=\alpha_{i+1}^p=\beta^p=1 (i=1,2)\rangle}$.\\
(xii) $\Phi_2(31)=\langle{\alpha,\alpha_1,\alpha_2 \mid [\alpha_1,\alpha]=\alpha^{p^2}=\alpha_2, \alpha_1^p=\alpha_2^p=1\rangle}$.

If $p=2$, then $G$ is isomorphic to\\
(xiii) $D_8 \times \mathbb{Z}_2^{(4)}$, \\
(xiv) $\langle a,b,c,d,e \mid [d, c] = [e, b] = a^2, a^4=b^2=c^2=d^2=e^2=1\rangle$,\\
(xv) $\mathbb{Z}_2 \times \langle a,b,c,d \mid  [b,c]=[a,d]=b^2,a^2=b^4=c^2=d^2=1\rangle$,\\
(xvi) $\mathbb{Z}_2 \times \langle a,b,c,d \mid b^2=c^2, [b,c]=[a,d]=b^2,a^2=b^4=c^4=d^2=1\rangle$,\\ 
(xvii) $\mathbb{Z}_2 \times \mathbb{Z}_2 \times X $ where $X=\langle{a,b,c \mid [b,c]=a^2,[a,b]=[a,c]=1,a^4=b^2=c^2=1\rangle}$\\
(xviii) $Q_8 \times \mathbb{Z}_2^{(3)}$, \\
(xix) $\mathbb{Z}_2^{(4)} \rtimes \mathbb{Z}_2$, \\
(xx) $\langle a, b, c \mid  [a, c] = [b, c] = 1, (ba)^2 = (ab)^2, ba^2 = a^2b, a^4 = b^2 = c^2 = 1 \rangle$,\\
(xxi) $QD_{16}$, \\
(xxii) $Q_{16}$, \\
(xxiii) $\langle{a,b \mid [a,b]=a^4,a^8=b^2=1\rangle}$,\\ 
(xxiv) $\langle a,b,c \mid [a, c] = a^2, [b, c] = b^2,a^4 = b^4=c^2 = 1 \rangle$,\\
where $D_n$ denotes Dihedral group of order $n$, $QD_n$ denotes QuasiDihedral groups of order n and $Q_n$ denotes Quaternion  group of order $n$.  
\end{thm}
\section{Preliminaries}
In this section we list following results which are used in our proof.
\begin{thm}(see \cite[Theorem 4.1]{MRRR})\label{J}
Let $G$ be a finite group and $K$ a central subgroup. Set $A = G/K$. Then
$|M(G)||G'\cap K|$ divides $|M(A)| |M(K)| |A \otimes K|$
\end{thm}
\begin{prop}(see \cite[Proposition 2.4]{MRR})\label{J1}
Let $G$ be a finite nilpotent group of class $c \geq 2$. Then
$|\gamma_c(G)||M(G)| \leq |M(G/\gamma_c(G))||G/Z_{c-1}(G) \otimes \gamma_c(G)|$
\end{prop}
\begin{thm}(see \cite[Corollary 4.16]{BT})\label{SHH}
Let $G$ be an extra special p-group of order $p^{2n+1}$.\\
(i) If $n \geq 2$, then $|M(G)| = p^{2n^2-n-1}.$\\
(ii) If $n = 1$, then the Schur multiplier of $D_8 , Q_8 , ES_p(p^3)$ and $ES_{p^2}(p^3)$ are isomorphic to $\mathbb{Z}_2, 1, \mathbb{Z}_p \times \mathbb{Z}_p$ and $1$ respectively.
\end{thm}
The following result follows from \cite{KO} for $|G'|=p$ and from \cite[page. 4177]{EG} for $|G'|=p^2$. 
\begin{thm}\label{SHHH}
For non-abelian $p$-groups $G$ of order $p^4$ with $|G'|=p$, 
$M(\Phi_2(211)a) \cong \mathbb{Z}_p \times \mathbb{Z}_p$, $M(\Phi_2(1^4)) \cong \mathbb{Z}_p^{(4)}$, $M(\phi_2(31)) \cong {1}$, $M(\Phi_2(22)) \cong \mathbb{Z}_p$, $M(\Phi_2(211)b) \cong \mathbb{Z}_p \times \mathbb{Z}_p$, $M(\Phi_2(211)c) \cong \mathbb{Z}_p \times \mathbb{Z}_p$ \\
and For $|G'|=p^2$,\\
$M(\Phi_3(211)a) \cong \mathbb{Z}_p$, $M(\Phi_3(211)b_r) \cong \mathbb{Z}_p$, $M(\Phi_3(1^4)) \cong \mathbb{Z}_p \times \mathbb{Z}_p$
\end{thm}
The following three lemmas follow from \cite[Main theorem]{PN}.
\begin{lemma}\label{5}
Let $G$ be a non-abelian $p$-group of order $p^n$ with $|M(G)|=p^{\frac{1}{2}n(n-1)-6}$. Then $n \leq 8$.
\end{lemma}
\begin{lemma}\label{6}
Let $G$ be a non-abelian $p$-group of order $p^8$ with $|M(G)|=p^{\frac{1}{2}n(n-1)-6}$. Then $G \cong ES_p(p^3) \times \mathbb{Z}_p^{(5)}$.
\end{lemma}
\begin{lemma}\label{m3}
There is no $p$-group $G$ of order $p^n(n \geq 6)$ with $|G'| \geq p^3$ and $|M(G)|=p^{\frac{1}{2}n(n-1)-6}$.
\end{lemma}
\begin{lemma}\label{7}
There is no non-abelian $p$-group $G$ of order $p^7$ with $|M(G)|=p^{\frac{1}{2}7(7-1)-6}=p^{15}$.
\end{lemma}
\begin{proof}
It follows from \cite[Theorem 21]{PN2}. \hspace{5.4cm} $\hfill\square$
\end{proof}  
\begin{lemma}\label{8}
Let $G$ be a non-abelian $p$-group of order $p^4$ with $|M(G)|=p^{\frac{1}{2}4(4-1)-6}=1$. Then $G \cong \Phi_2(31)$.
\end{lemma}
\begin{proof}
This result follows from Theorem \ref{SHHH}. \hspace{5 cm} $\hfill\square$
\end{proof} 
Hence from the discussion above it is clear that we have to study groups of order $p^5$ and $p^6$. 
\section{Groups of order $p^5$}
In this section we characterize groups of order $p^5$ which have Schur multiplier of order $p^4$. 
\begin{lemma}
There is no non-abelian $p$-group of order $p^5$ with $|G'|= p^3$ and $|M(G)|=p^{\frac{1}{2}5(5-1)-6}=p^4$.
\end{lemma}
\begin{proof}
Notice that nilpotency class of $G$ is either $3$ or $4$.

Let $G$ be of class $4$. Then $G$ lie in the isoclinism classes $\Phi_9$ or $\Phi_{10}$. By Proposition \ref{J1} and Theorem \ref{SHHH} we have $|M(G)| \leq p^3$.

Now assume that the nilpotency class of $G$ is 3. Then it follows that $G$ lie in $\Phi_6$ in \cite{RJ}. In this case take any subgroup $K \subset Z(G) \cap G'$ of order $p$. By Theorem \ref{J} we have $|M(G)|p \leq |M(G/K)|p^2$. Here $G/K$ is of order $p^4$ with $|(G/K)'|=p^2$. 
Hence it follows from Theorem \ref{SHHH} that $|M(G)| \leq p^3$. This concludes the proof.  $\hfill  \square$
\end{proof} 
Before we proceed to the next result, we explain a method by Blackburn and Evens \cite{BE} for computing Schur multiplier of $p$-groups of class $2$ with $G/G'$ is elementary abelian.

Here $G/G'$ and $G'$ are elementary abelian of order $p^3$ and $p^2$ respectively. We can consider $G/G'$ and $G'$ as vector spaces over $GF(p)$, denote by $V, W$ respectively. Let $v_1,v_2 \in V$ such that $v_i=g_iG', i \in \{1,2\}$ and take $(v_1,v_2)=[g_1,g_2]$.

Let $X_1$ be the subspace of $V \otimes W$ spanned by all \\
\centerline{$v_1 \otimes (v_2,v_3) + v_2 \otimes (v_3,v_1) + v_3 \otimes (v_1,v_2)$} 

Consider a map $f:V \rightarrow W$ given by $f(gG')=g^p$, $g \in G$. We denote by $X_2$ the subspace spanned by all $v \otimes f(v)$ for $v \in V$ and take $X=X_1+X_2$. Now consider a homomorpism $\sigma: V\wedge V \rightarrow V \otimes W/X$ given by\\
$\sigma(v_1 \wedge v_2)=v_1 \otimes f(v_2)+ {p \choose 2}v_2 \otimes (v_1,v_2)+X$.
Then there exists an abelian group $M^*$ having a subgroup $N$ for which \\
\centerline{$1 \rightarrow V\otimes W/X \rightarrow M^* \xrightarrow{\xi} V \wedge V \rightarrow 1$}\\
is exact, where $N \cong V \otimes W/X, M^*/N \cong V \wedge V$  and $(\sigma\xi)(m)=m^p$ for all $m \in M^*$.
\begin{thm}\label{B}\cite{BE}
With the notation above, consider a homomorphism $\rho:V \wedge V \rightarrow W$ given by\\
\centerline{$\rho(v_1 \wedge v_2)=(v_1,v_2)$ for all $v_1,v_2 \in V$.}\\
Denote by $M$, the subgroup of $M^*$ containg $N$ for which $M/N$ corresponds to $ker \rho$. Then $M(G) \cong M$
\end{thm}
\begin{lemma}\label{2}
Let $G$ be a non-abelian $p$-group of order $p^5$ with $|G'|=p^2$ and $|M(G)|=p^{\frac{1}{2}5(5-1)-6}=p^4$. Then $G \cong \Phi_3(1^5), \Phi_7(1^5)$.
\end{lemma}
\begin{proof}
Groups of order $p^5$ with $|G'|=p^2$ lie in the isoclinism classes $\Phi_3, \Phi_4, \Phi_7$ or $\Phi_8$ \cite{RJ}.
  
In isoclinism class $\Phi_3$, some groups are direct product of its subgroups. So for them we can easily compute $M(G)$ and we see $|M(\Phi_3(1^5))|=p^4$.  Now we consider other groups. We have $|M(G)|p \leq |M(G/\gamma_3(G))|p^2$ using Proposition \ref{J1}. Observe that $|M(G/\gamma_3(G))| \leq p^2$ except $\Phi_3(2111)c$. Therefore $|M(G)| \leq p^3$ except $\Phi_3(2111)c$. For $G \cong \Phi_3(2111)c$ by Theorem \ref{J} (taking $K=Z(G)$) it follows that $|M(G)| \leq p^3$.

In class $\Phi_4$, every group $G$ has nilpotency class 2 with $G/G'$ elementary abelian. By Theorem \ref{B} $|M(G)|/|N|=|V \wedge V|/|W|$ and thus one can show that $G=\Phi_4(1^5)$ has $|M(G)|=p^6$ but for all other groups $G$ in $\Phi_4$, $|M(G)| \leq p^3$.

In class $\Phi_7$, for the groups $G \cong \Phi_7(2111)a, \Phi_7(2111)b_r,\Phi_7(2111)c$ we can choose a normal subgroup $K$ such that $G/K$ is cyclic with $|M(K)|=p, |K'|=p^2$. Hence $|M(G)| \leq p^3$ by \cite[Theorem 3.1]{MRR}.
Now by \cite[Theorem 2.3.10]{GK} we have the following sequence is exact.\\
\centerline{$\Hom(Z, \mathbb{C}^*) \xrightarrow{Tra} M(G/Z) \xrightarrow{inf} M(G) \xrightarrow{\delta} G/G' \otimes Z$}\\
In particular for $G \cong \Phi_7(1^5)$, taking $Z=Z(G)$, $Im(Tra) \cong ker(Inf) \cong G' \cap Z \cong \mathbb{Z}_p$. We can see in \cite{RJ}, $G$ is capable as $E/Z(E) \cong G$ for groups $E$ of order $p^6$ in isoclinism class $\Phi_{30}$. Thus it follows from  \cite[Corollary 2.5.8 and 2.5.10]{GK} that $\delta$ is not trivial map. So $|ker(\delta)|=|Im(Inf)|=\frac{|M(G/Z)|}{p}=p^3 < |M(G)|$. So $p^4 \leq |M(G)|$ and by \cite[Theorem 3.1]{MRR} (taking $K=<\alpha,\alpha_1, \alpha_2,\alpha_3>$) $|M(G)| \leq p^4$.\\
Hence we conclude $|M(\Phi_7(1^5))|=p^4$.

Finally consider the class $\Phi_8$, consisting of only one group $\Phi_8(32)$. Observe that $|M(\Phi_8(32))| \leq p^3$ by Theorem \ref{J} (taking $K=Z(G)$). \hspace{4cm} $\hfill\square$
\end{proof} 
\begin{lemma}\label{1}
Let $G$ be a non-abelian $p$-group of order $p^5$ with $|G'|=p$ and $|M(G)|=p^{\frac{1}{2}5(5-1)-6}=p^4$. Then $G \cong \Phi_2(2111)c, \Phi_2(2111)d$.
\end{lemma}
\begin{proof}
Groups of order $p^5$ with $|G'|=p$ lie in the isoclinism classes $\Phi_2$ and $\Phi_5$ \cite{RJ}.\\ The isoclinism class $\Phi_5$ consists of extra-special $p$-groups which have Schur multiplier of order $p^5$ by Theorem \ref{SHH}. Now we consider the class $\Phi_2$. Since $|M(G)|=p^4$, by Theorem \ref{J} we have $G/G' \cong \mathbb{Z}_{p^2} \times \mathbb{Z}_p \times \mathbb{Z}_p$ or $\mathbb{Z}_p^{(4)}$. Now it follows that $G$ is either isomorphic to $\Phi_2(311)b, \Phi_2(221)c$ or $G$ is direct product of its subgroups.

If $G \cong \Phi_2(311)b, \Phi_2(221)c$, then by Theorem \ref{J} $|M(G)| \leq p^3$ for suitable $K$. If $G$ is direct product of its subgroups, then $|M(G)|=p^4$ only for $\Phi_2(2111)c, \Phi_2(2111)d$.  $\hfill\square$
\end{proof}   
\section{Groups of order $p^6$ and Proof of Main theorem}
In this section we characterize groups of order $p^6$ having Schur multiplier of order $p^9$ and prove the Main Theorem.
\begin{lemma}
Let $G$ be a non-abelian $p$-group of order $p^6$ with $|M(G)|=p^{\frac{1}{2}6(6-1)-6}=p^9$. Then $G/G'$ is elementary abelian.
\end{lemma}
\begin{proof}
From Lemma \ref{m3} we have $|G'| \leq p^2$. It follows from \cite{RJ} that if $|G'|=p$, then $G/Z(G)$ is isomorphic to $\mathbb{Z}_p \times \mathbb{Z}_p$ or $\mathbb{Z}_p^{(4)}$. If $G' \cong \mathbb{Z}_p \times \mathbb{Z}_p$ then $G/Z(G)$ is isomorphic to $ES_p(p^3), \mathbb{Z}_p^{(3)}, ES_p(p^3) \times \mathbb{Z}_p$, $ES_p(p^3) \times \mathbb{Z}_p^{(2)}$ or $\mathbb{Z}_p^{(4)}$. If $G' \cong \mathbb{Z}_{p^2}$, then again by \cite{RJ} $G/Z(G)$ is isomorphic to $\Phi_2(22)$ or $\mathbb{Z}_{p^2} \times \mathbb{Z}_{p^2}$.
In all the above cases using \cite[Proposition 1]{EW} we conclude that if $G/G'$ is not elementary abelian, then $|M(G)| < p^9$, which is not our case. Hence the result follows. \hspace{4 cm} $\hfill\square$ 
\end{proof}  
\begin{lemma}\label{3}
Let $G$ be a non-abelian $p$-group of order $p^6$ with $|G'|=p$ and $|M(G)|=p^9$, then $G$ is isomorphic to one of the following: $\Phi_2(21^4)a, \Phi_2(21^4)b, $ $\Phi_5(21^4)a, \Phi_5(1^6), \Phi_5(21^4)b$.
\end{lemma}
\begin{proof}
It follows from \cite{RJ} that $G$ lie in the isoclinism classes $\Phi_2$ or $\Phi_5$. By the preceeding lemma we have $G/G'$ is elementary abelian of order $p^5$. So in the isoclinism class $\Phi_2$ we have to check $\Phi_2(21^4)a, \Phi_2(21^4)b, \Phi_2(21^4)c, \Phi_2(21^4)d, \Phi_2(1^6)$ and in the isoclinism class $\Phi_5$ we have to check $\Phi_5(21^4)a,  \Phi_5(1^6), \Phi_5(21^4)b$.

If $G \cong \Phi_5(21^4)b$, then by Theorem \ref{J} (taking $K=Z(G)$) we have $|M(G)| \leq p^9$ and by \cite[Corollary 3.2]{MRRR} $|M(G)| \geq p^9$. Hence $\Phi_5(21^4)b$ has Schur multiplier of order $p^9$. 

All other above groups are direct product of its subgroups. So it is easy to see that among them $\Phi_2(21^4)a, \Phi_2(21^4)b, \Phi_5(21^4)a, \Phi_5(1^6)$ have Schur multiplier of order $p^9$.  \hspace{10.3 cm}$\hfill\square$
\end{proof} 
\begin{lemma}
There is no non-abelian $p$-group of order $p^6$ with $G'\cong \mathbb{Z}_{p^2}$ and $|M(G)|=p^9$.
\end{lemma}
\begin{proof}
From \cite{RJ} it suffices to consider isoclinism classes $\Phi_8, \Phi_{14}$. Now using Theorem \ref{J} (taking $K=Z(G)$) we can easily show that $|M(G)| < p^9$. \hspace{2 cm}$\hfill\square$
\end{proof}
\begin{lemma}
Let $G$ be a non-abelian $p$-group of order $p^6$ with $G' \cong  \mathbb{Z}_p \times  \mathbb{Z}_p $ and $|M(G)|=p^9$. Then $Z(G)$ is of exponent $p$ and $G'\subseteq Z(G)$.
\end{lemma}
\begin{proof}
Let the exponent of $Z(G)$ be $p^k(k \geq 2)$ and $K$ be a cyclic central subgroup  of order $p^k$. Then using Theorem \ref{J} and \cite{PN}, we have\\
\centerline{$|M(G)| \leq p^{-1}|M(G/K)||G/K \otimes K| \leq p^{-1}p^{\frac{1}{2}(n-3)(n-4)+1}p^{(n-3)}$}. \\
When $n=6$ it gives $|M(G)| < p^9$, which is a contradiction. Hence $Z(G)$ is of exponent $p$. 

Now assume that $G' \nsubseteq Z(G)$. Then it follows that $|Z(G)| \leq p^3$.
If $|Z(G)|=p$, then $G$ is of nilpotency class 3. Now by Proposition \ref{J1}, we get $|M(G)| < p^9$. If $|Z(G)|=p^2$, then by the assumption there is a central subgroup $K$ of order $p$ such that $G' \cap K=1$ and $(G/K)'=p^2$. By Theorem \ref{J}, $|M(G)| \leq |M(G/K)|p^3$. This is possible only when $|M(G/K)| \geq p^6$. Now by Theorem \cite{PN, PN3} we see that there is no such $G/K$ of order $p^5$ such that $|M(G/K)| \geq p^6$ with $(G/K)'=p^2$ and $Z(G/K)=p$. Finally if $|Z(G)|=p^3$, then consider a central subgroup of order $p^2$ such that $G' \cap K=1$. Then $|M(G)| \leq |M(G/K)|p^4$ by Theorem \ref{J}. Therefore it follows from Theorem \ref{SHHH} $|M(G)| < p^9$. \hspace{5cm} $\hfill\square$
\end{proof} 
\begin{lemma}\label{4}
Let $G$ be a non-abelian $p$-group of order $p^6$ with $G'\cong  \mathbb{Z}_p \times  \mathbb{Z}_p$ and $|M(G)|=p^{\frac{1}{2}6(6-1)-6}=p^9$. Then $G$ is isomorphic to $\Phi_4(1^6)$.
\end{lemma}
\begin{proof}
By preceeding lemma, $G' \subseteq Z(G)$. 
Now consider a central subgroup $K$ in $G'$ of order $p$. By Theorem \ref{J} we have $|M(G)| p \leq |M(G/K)|p^4$. This is possible only when $|M(G/K)| \geq p^6$. Since $(G/K)' \cong \mathbb{Z}_p$ so $G/K \cong ES_p(p^3) \times \mathbb{Z}_p^{(2)}$ by \cite{PN}.
This tells that $G$ is of exponent $p$. So we have to study the groups of exponent $p$ in the isoclinism classes $\Phi_3, \Phi_4, \Phi_7, \Phi_{22}$ which are $\Phi_3(1^6), \Phi_4(1^6), \Phi_7(1^6), \Phi_{22}(1^6)$.

If $G \cong \Phi_{22}(1^6)$, then by \cite[Theorem 3.1]{MRR} (taking $K=<\alpha,\alpha_1,\alpha_2, \alpha_3,\beta_1>$) we have $|M(G)| < p^9$. Other groups are direct product of its subgroups. Hence it is easy to see that only $\Phi_4(1^6)$ has Schur multiplier of order $p^9$.  \hspace{3 cm}$\hfill\square$
\end{proof} 

We are now ready to prove our Main Theorem.
\subsection{Proof of Main Theorem}
Let $G$ be a group of order $p^n$ ($p$ odd) with $|M(G)|=p^{\frac{1}{2}n(n-1)-6}$. By Lemma \ref{5}, it follows that $n \leq 8$. If $|G|=p^8$, then the assertion $(i)$  of the main theorem follows from Lemma \ref{6}. By Lemma \ref{7} it follows that there is no non-abelian $p$-group of order $p^7$ with $|M(G)|=p^{\frac{1}{2}7(7-1)-6}=p^{15}$. So the problem reduces to studying groups of order $p^4$, $p^5$ and $p^6$.
If $|G|=p^4$, then the assertion $(xii)$ follows from Lemma \ref{8}.
Now we assume $|G|=p^5$. Then assertions $(viii), (ix)$ follows from Lemma \ref{1} and $(x), (xi)$  follows from Lemma \ref{2}. Now consider groups of order $p^6$. Then the assertions $(ii), (iii), (iv), (v), (vi)$ follow from Lemma \ref{3} and  $(vii)$ follows from \ref{4}.

For the case $p=2$, the classification follows from computation using HAP package \cite{HAP} of GAP \cite{GAP}.   $\hfill\square$
\\
\\
{\bf Acknowledgement}:
I am grateful to my supervisor Manoj K. Yadav for his guidance, motivation and discussions.  I wish to thank the Harish-Chandra Research Institute, the Dept. of Atomic Energy, Govt. of India, for providing excellent research facility.


\begin{thebibliography}{999}
\bibitem{IS1}
Schur, I, \emph{Über die Darstellung der endlichen Gruppen durch gebrochene lineare Substitutionen}, J. Reine Angew. Math. {\bf 127}, (1904), 20-50.

\bibitem{IS2}
Schur, I, \emph{Untersuchungen über die Darstellung der endlichen Gruppen durch gebrochene lineare Substitutionen}, J. Reine Angew. Math. {\bf 132}, (1907), 85-137.

\bibitem{JG}
Green, J.A., \emph{On the number of automorphisms of a finite group}, Proc. Roy. Soc. A {\bf 237}, (1956), 574-581. 

\bibitem{BY}
Berkovich, Ya. G., \emph{On the order of the commutator subgroups and the Schur multiplier of a finite p-group.}, J. Algebra. {\bf 34} (1991), 613-637.

\bibitem{ZH}
Zhou, X., \emph{On the order of Schur multipliers of finite p-groups,} Comm. Algebra, {\bf 22}, (1994), 1-8.

\bibitem{EG}
Ellis, G.,\emph{On the Schur multiplier of p-groups.},  Comm. Algebra, (1999), {\bf 27}(9): 4173-4177.

\bibitem{PN}
Niroomand, P, \emph{On the order of Schur multiplier of non-abelian p groups}, Journal of Algebra,  {\bf 322}, (2009), 4479-4482.

\bibitem{PN3}
Niroomand, P., \emph{Characterizing finite p-groups by their Schur multipliers}, C. R. Math. Acad. Sci. Paris, Ser. {\bf 350}, (2012), 867-870.

\bibitem{PN1}
 Niroomand, P., \emph{Characterizing finite p-groups by their Schur multipliers, $t(G) = 5$}, http://arxiv.org/abs/1001.4257.

\bibitem{PN2}
Niroomand, P., \emph{A note on the Schur multiplier of groups of prime power order}, Ricerche di Matematica, (2012), {\bf 61}, 341-346

\bibitem{EW}
Ellis, G., Wiegold, J., \emph{A bound on the Schur multiplier of a prime-power group}, Bull. Austral.math.soc, {\bf 60} (1999), 191-196.

\bibitem{SHJ}
Jafari, S.H., \emph{Finite $p$-groups whose order of their Schur multiplier is
given$(t=6)$}, The Extended Abstract of the 6th International Group Theory Conference (2014), 92-95.

\bibitem{MJ}
Jones, M.R., \emph{Multiplicators of p-Groups}, Math. Z {\bf 127} (1972),  
165-166.

\bibitem{MRR}
Jones, M.R., \emph{Some inequalities for multiplicators of a finite group II}, Proc. Amer. Math. Soc. {\bf 45}(2) (1974), 167-172.

\bibitem{MRRR}
Jones, M.R., \emph{Some inequalities for multiplicators of a finite group}, Proc. Amer. Math. Soc. {\bf 39}(3) (1973b), 450-456.

\bibitem{BE}
Blackburn, N., Evens, L., \emph{Schur multipliers of p-groups}, J. Reine Angew. Math. {\bf 309}, (1979), 100-113.

\bibitem{KO}
Kim Seon Ok, \emph{On p-groups of order $p^4$}, Comm. Korean Math. Soc.{\bf 16} (2001), 205-210. 

\bibitem{GK}
Karpilovsky, G., \emph{The Schur multiplier}. London Math. Soc. Monographs, Oxford Univ. Press, 1987.

\bibitem{RJ}
James, R., \emph{The groups of order $p^6$}(p an odd prime), Math. Comp. {\bf 34} (1980), 613-637.

\bibitem{BT}
Beyl, F.R and Tappe, J. \emph{Group extensions, representations and the Schur multiplicator}, Lecture Notes in Mathematics No. 958, Springer-Verlag, Berlin. 1982.

\bibitem{HAP}
Ellis, G. (2008).\emph{HAP-Homological algebra programming} A refreed GAP 4 package (GAP Group 2005), available at http://hamilton.nuigalway.ie/Hap/www.

\bibitem{GAP}
The GAP Group. (2005). \emph{GAP–Groups Algorithms and Programming.} version 4.4. Available at: http://www.gap-system.org.

\end{thebibliography}
\end{document}